\documentclass[reqno,11pt]{amsart}
\usepackage{amscd,amssymb,verbatim,array}
\usepackage{hyperref}

\numberwithin{equation}{section} \theoremstyle{plain}

\newcommand{\Complex}{\mathbb C}
\newcommand{\Real}{\mathbb R}
\newcommand{\N}{\mathbb N}
\newcommand{\ddbar}{\overline\partial}
\newcommand{\pr}{\partial}
\newcommand{\ol}{\overline}

\newcommand{\set}[1]{\left\{#1\right\}}
\newcommand{\To}{\rightarrow}

\newtheorem{theorem}{Theorem}[section]
\newtheorem{lemma}[theorem]{Lemma}
\newtheorem{proposition}[theorem]{Proposition}

\newtheorem{definition}[theorem]{Definition}
\newtheorem{ass}[theorem]{Assumption}

\theoremstyle{definition}

\theoremstyle{remark}

\numberwithin{equation}{section}

\newcommand{\abs}[1]{\lvert#1\rvert}

\begin{document}

\title[On the coefficients of the equivariant Szeg\H{o} kernel asymptotic expansions]
{On the coefficients of the equivariant Szeg\H{o} kernel asymptotic expansions}

\author{Chin-Yu Hsiao}

\address{Institute of Mathematics, Academia Sinica , Astronomy-Mathematics Building, No. 1, Sec. 4, Roosevelt Road, Taipei 10617, Taiwan}
\thanks{The first author was partially supported by Taiwan Ministry of Science and Technology projects 108-2115-M-001-012-MY5 and 109-2923-M-001-010-MY4 and Academia Sinica Career Development Award. The second author was supported by Taiwan Ministry of Science and Technology projects 107-2115-M-008-007-MY2 and 109-2115-M-008-007-MY2.}

\email{chsiao@math.sinica.edu.tw or chinyu.hsiao@gmail.com}

\author{Rung-Tzung Huang}
\address{Department of Mathematics, National Central University, Chung-Li, Taoyuan 32001, Taiwan}

\email{rthuang@math.ncu.edu.tw}

\author{Guokuan Shao}

\address{School of Mathematics (Zhuhai), Sun Yat-sen University, Zhuhai 519082, Guangdong, China}

\email{shaogk@mail.sysu.edu.cn}


\begin{abstract}
Let $(X, T^{1,0}X)$ be a compact connected orientable strongly pseudoconvex CR manifold of dimension $2n+1$, $n\geq1$. Assume that $X$ admits a connected compact Lie group $G$ action and a transversal CR $S^1$ action,  we compute the coefficients of the first two lower order terms of the equivariant Szeg\H{o} kernel asymptotic expansions with respect to the $S^1$ action. 
\end{abstract}

\maketitle \tableofcontents


\section{Introduction and statement of the main results}\label{s-gue170124}

Let $(X, T^{1,0}X)$ be a CR manifold of dimension $2n+1$, $n\geq1$, and $\Box^{(q)}_b$  the Kohn Lalpacian acting on $(0,q)$ forms. The Szeg\H{o} kernel $S^{(q)}(x,y)$ is the distribution kernel of the orthogonal projection $S^{(q)}:L^2_{(0,q)}(X)\To {\rm Ker\,}\Box^{(q)}_b$. The study of Szeg\H{o} kernels is an important subject in several complex variables and CR geometry. Assume that $X$ is the boundary of a strongly pseudoconvex domain, Boutet de Monvel-Sj\"ostrand~\cite{BouSj76} proved that $S^{(0)}(x,y)$
is a complex Fourier integral operator. 
The first author \cite{Hsiao08} established Boutet de Monvel-Sj\"ostrand type theorems for $S^{(q)}(x,y)$, $q>0$, on a non-degenerate CR manifold.

When a CR manifold admits a compact Lie group $G$ action, the study of  $G$-equivariant Szeg\H{o} kernels is important in geometric quantization theory. 
Recently, Hsiao-Huang \cite{HH} obtained $G$-invariant Boutet de Monvel-Sj\"ostrand type theorems and Hsiao-Ma-Marinescu~\cite{HMM} established geometric quantization on CR manifolds by using  $G$-invariant Szeg\H{o} kernels asymptotic expansions. 
In this paper, we consider a strongly pseudoconvex CR manifold $X$ which admits a compact Lie group $G\times S^1$ action. Under certain assumptions of the group $G\times S^1$ action, by using the method in \cite{HH}, we can show that the $m$-th Fourier component of the $G$-equivariant Szeg\H{o} kernel $S_{k,m}:=S^{(0)}_{k,m}$ admits an asymptotic expansion in $m$. It is quite interesting to know the coefficients of the asymptotic expansion for $S_{k,m}$. In this work, we calculated the first two coefficients of the expansion of $S_{k,m}$. It should be mentioned that the coefficients of the asymptotic expansion for $S_{k,m}$ will be used in the study of $G$-equivariant Toeplitz operator. 

The calculation of coefficients of asymptotic expansions for Bergamn and Szeg\H{o} kernels is an active subject. Lu \cite{Lu} computed the coefficients of the first four lower order terms for Bergman kernel asymptotic expansions by using Tian's peak section method, see also Wang \cite{W}. Ma-Marinescu \cite{MM08a} used formal power series to compute coefficients for Bergman kernel asymptotic expansions, see also \cite{HHL18, Hsiao12, Hsiao16, LuW, MM06, MM12}. Ma-Zhang \cite{MZ} studied asymptotic expansions of $G$-invariant Bergman kernels and also computed the first two coefficients of the expansion. 


We now formulate the main results. We refer to Section~\ref{s:prelim} for some notations and terminology used here. Let $(X, T^{1,0}X)$ be a compact connected orientable strongly pseudoconvex CR manifold of dimension $2n+1$, $n\geq1$, where $T^{1,0}X$ denotes the CR structure of $X$. Assume that $X$ admits a $d$-dimensional connected compact Lie group $G$ action and a transversal CR $S^1$ action. Let $T$ be the global real vector field on $X$ induced by the $S^1$ action and let 
$\omega_0$ be the global one form given by \eqref{e-gue200803yydI} below. Let $J$ be the complex structure on $HX$ defined in the beginning of Section~\ref{s-gue200807yyd}.
Denote by $\mathfrak{g}$ the Lie algebra of $G$. For any $\xi \in \mathfrak{g}$, let $\xi_X$ be the vector field on $X$ induced by $\xi$. That is, $(\xi_X u)(x)=\frac{\partial}{\partial t}\left(u(\exp(t\xi)\circ x)\right)|_{t=0}$, for any $u\in C^\infty(X)$. Let $\underline{\mathfrak{g}}={\rm Span\,}(\xi_X;\, \xi\in\mathfrak{g})$.
We assume throughout the paper that

\begin{ass}\label{a-gue170123I}
\begin{equation}\label{e-gue200726yyd}
\mbox{The Lie group $G$ action is CR and preserves $\omega_0$ and $J$ and },
\end{equation}
\begin{equation}\label{e-gue170111ryI}
	\mbox{$T$ is transversal to the space $\underline{\mathfrak{g}}$ at every point $p\in\mu^{-1}(0)$},
	\end{equation}
	\begin{equation}\label{e-gue170111ryII}
	\mbox{$e^{i\theta}\circ g\circ x=g\circ e^{i\theta}\circ x$, for all $x\in X$, $\theta\in[0,2\pi[$, $g\in G$}. 
	\end{equation}
\end{ass}

We recall that the Lie group $G$ action preserves $\omega_0$ and $J$  means that 
$g^\ast\omega_0=\omega_0$ on $X$ and $g_\ast J=Jg_\ast$ on $HX$, for every $g\in G$, where $g^*$ and $g_*$ denote  the pull-back map and push-forward map of $G$, respectively. The $G$ action is CR means that for every $\xi_X\in\underline{\mathfrak{g}}$, 
	\begin{equation*}
		[\xi_X, C^{\infty}(X,T^{1,0}X)]\subset C^{\infty}(X,T^{1,0}X).
	\end{equation*}
\begin{definition}\label{d-gue170124}
	The momentum map associated to the form $\omega_0$ is the map $\mu:X \to \mathfrak{g}^*$ such that, for all $x \in X$ and $\xi \in \mathfrak{g}$, we have 
	\begin{equation*}\label{E:cmpm}
	\langle \mu(x), \xi \rangle = \omega_0(\xi_X(x)).
	\end{equation*}
\end{definition}

We also assume that 

\begin{ass}\label{a-gue170123II}
	$0$ is a regular value of $\mu$ and $G\times S^1$ acts freely near $\mu^{-1}(0)$. 
\end{ass}

By Assumption~\ref{a-gue170123II}, $\mu^{-1}(0)$ is a $d$-codimensional submanifold of $X$. In \cite{HH}, it was showed that $\mu^{-1}(0)/G$ is a CR manifold with natural CR structure induced by $T^{1,0}X$ of dimension $2n-2d+1$. 

Let $R=\{R_1,R_2,...\}$ be the collection of all irreducible unitary representations of $G$, including only one representation from each equivalent class. 
Write
\begin{equation*}\label{e-9291}
\begin{split}
R_k:G&\to GL(\mathbb{C}^{d_k}), \ \ d_k<\infty,\\
g&\to (R_{k,j,l}(g))_{j,l=1}^{d_k},
\end{split}
\end{equation*}
where $d_k$ is the dimension of the representation $R_k$.
Denote by $\chi_k(g):=\text{Tr}R_k(g)$ the trace of the matrix $R_k(g)$ (the character of $R_k$). 
Let $u\in\Omega^{0,q}(X)$. For every $k=1,2,\ldots$, define 
\begin{equation}\label{e-9292}
u_k(x)=d_k\int_G (g^\star u)(x)\overline{\chi_k(g)}d\mu(g),
\end{equation}
where $d\mu(g)$ is the probability Haar measure on $G$.
For every $k=1,2,\ldots$, 
set
\begin{equation}\label{e-9293}
\Omega^{0,q}(X)_k:=\{u(x)\in \Omega^{0,q}(X)| u(x)=u_k(x)\}.
\end{equation} 

The Levi form on $X$ induces a Hermitian metric $\langle\,\cdot\,|\,\cdot\,\rangle$ on $\Complex TX$ as follows: 
\begin{equation}\label{e-gue200726yyda}
\begin{split}
&\langle\,u\,|\,v\,\rangle=\frac{-1}{2i}\langle\,d\omega_0\,,\,u\wedge\ol v\,\rangle,\  \ u, v\in T^{1,0}X,\\
&T^{1,0}X\perp T^{0,1}X,\\
&\langle\,T\,|\,T\,\rangle=1,\ \ T\perp(T^{1,0}X\oplus T^{0,1}X).
\end{split}
\end{equation}
Let $(\,\cdot\,|\,\cdot\,)$ be the $L^2$ inner product on $\Omega^{0,q}(X)$ induced by $\langle\,\cdot\,|\,\cdot\,\rangle$. 
For every $m\in\mathbb Z$, $k=1,2,\ldots$, let
\begin{equation}\label{e-gue150508dIm}
\begin{split}
&\Omega^{0,q}_m(X):=\set{u\in\Omega^{0,q}(X);\, Tu=imu},\ \ q=0,1,2,\ldots,n,\\
&\Omega^{0,q}_{m}(X)_k=\set{u\in\Omega^{0,q}(X)_k;\, Tu=imu},\ \ q=0,1,2,\ldots,n.
\end{split}
\end{equation}
Let $L^2_{(0,q)}(X)_k$, $L^2_{(0,q),m}(X)$ and $L^2_{(0,q),m}(X)_k$ be the completions of $\Omega^{0,q}(X)_k$, $\Omega^{0,q}_m(X)$ and $\Omega^{0,q}_m(X)_k$ with respect to $(\,\cdot\,|\,\cdot\,)$ respectively. 
We denote $C^\infty(X)_k:=\Omega^{0,0}(X)_k$, $L^2(X)_k:=L^2_{(0,0)}(X)_k$, $C^\infty_m(X):=\Omega^{0,0}_m(X)$, $L^2_m(X):=L^2_{(0,0),m}(X)$, $C^\infty_m(X)_k:=\Omega^{0,0}_m(X)_k$, $L^2_m(X)_k:=L^2_{(0,0),m}(X)_k$.

Let $\ddbar_b: \Omega^{0,q}(X)\To\Omega^{0,q+1}(X)$ be the tangential Cauchy-Riemann operator. Since $G$ and $S^1$ actions are CR, $\ddbar_b: \Omega^{0,q}_m(X)_k\To\Omega^{0,q+1}_m(X)_k$, for every $m\in\mathbb Z$, $k=1,2,\ldots$. Hence, 
\[\begin{split}
\mbox{$\ddbar_b: {\rm Dom\,}\ddbar_b\bigcap L^2_m(X)_k\To L^2_{(0,1),m}(X)_k$, for every $m\in\mathbb Z$ and $k=1,2,\ldots$}. 
\end{split}\]
Let 
\[S_{k,m}:=S^{(0)}_{k,m}: L^2(X)\To {\rm Ker\,}\ddbar_b\bigcap L^2_m(X)_k\]
be the orthogonal projection with respect to $(\,\cdot\,|\,\cdot\,)$. We call $S_{k,m}$ the $m$-th Fourier component of the $G$-equivariant Szeg\H{o} 
kernel. Since the $S^1$ action is transversal, it is easy to see that ${\rm Ker\,}\ddbar_b\bigcap L^2_m(X)_k$ is a finite dimensional subspace of $C^\infty_m(X)_k$ and hence $S_{k,m}$ is smoothing. Let $S_{k,m}(x,y)\in C^\infty(X\times X)$ be the distribution kernel of $S_{k,m}$. We can repeat the procedure as in~\cite{HH} and deduce that for every $x\in\mu^{-1}(0)$, 
\begin{equation}\label{e-gue200731yyd}
\begin{split}
S_{k,m}(x,x)\sim\sum^{+\infty}_{j=0}m^{n-\frac{d}{2}-j}b_{j,k}(x)\  \ \mbox{in $S^{n-\frac{d}{2}}_{{\rm loc\,}}(1; \mu^{-1}(0)\times\mu^{-1}(0))$},\\
b_{j,k}(x)\in C^\infty(\mu^{-1}(0)),\ \ j=0,1,\ldots. 
\end{split}
\end{equation}
(We refer the reader to Section~\ref{s-gue170111w} for the semi-classical notations used in \eqref{e-gue200731yyd} and Theorem~\ref{t-gue170128I} below.) The goal of this work is to compute the first two terms of the expansion \eqref{e-gue200731yyd}. More precisely, we have

\begin{theorem}\label{t-gue170128I}
	With the assumptions and notations used above, 
	let $p\in\mu^{-1}(0)$ and $U$ an open neighborhood of  $p$ with a local coordinate $x=(x_1,\ldots,x_{2n+1})$. 
	Then, as $m\To+\infty$, 
	\begin{equation}\label{e-gue170117pVIIIm}
	\begin{split}
	&S_{k,m}(x,y)=e^{im\Psi_k(x,y)}b_{k}(x,y,m)+O(m^{-\infty}),\\
	&b_{k}(x,y,m)\in S^{n-\frac{d}{2}}_{{\rm loc\,}}(1; U\times U),\\
	&\mbox{$b_{k}(x,y,m)\sim\sum^\infty_{j=0}m^{n-\frac{d}{2}-j}b_{j,k}(x,y)$ in $S^{n-\frac{d}{2}}_{{\rm loc\,}}(1; U\times U)$},\\
	&b_{j,k}(x,y)\in C^\infty(U\times U),\ \ j=0,1,2,\ldots,
	\end{split}
	\end{equation}
	$\Psi_k(x,y)\in C^\infty(U\times U)$, $d_x\Psi_k(x,x)=-d_y\Psi_k(x,x)=-\omega_0(x)$, for every $x\in\mu^{-1}(0)$, ${\rm Im\,}\Psi_k\geq0$, $\Psi_k(x,y)=0$ if and only if $x=y\in\mu^{-1}(0)$.
	Moreover, we have
	\begin{equation}\label{e-0526}
	b_{0,k}(p)=\frac{1}{2}\pi^{\frac{d}{2}-n-1}2^{\frac{d}{2}}\frac{d_k^2}{V_{{\rm eff\,}}(p)}.
	\end{equation}
	\begin{equation}\label{e-05261}
	\begin{split}
b_{1,k}(p)=&
\frac{1}{4}\pi^{\frac{d}{2}-n-1}2^{\frac{d}{2}}   \frac{d_k^2}{V_{{\rm eff\,}}(p)}R(p)\\
&+ \frac{1}{4}\pi^{\frac{d}{2}-n-1}2^{\frac{d}{2}}  \frac{d_k}{V_{{\rm eff\,}}(p)^{1+\frac{2}{d}}} \Delta_{d\mu}\bar\chi_k(p)\\
&-\frac{1}{32}\pi^{\frac{d}{2}-n-1} 2^{\frac{d}{2}} \frac{d_k^2}{V_{{\rm eff\,}}(p)}(2(V_{{\rm eff\,}}(p))^{-\frac{2}{d}}S_G(p)-R_e(p)),
	\end{split}
	\end{equation}
where $R$ is the Tanaka-Webster scalar curvature with respect to the pseudohermitian structure $-\omega_0$ (cf. \eqref{e-gue200825twsc}), $\chi_k$ is the character of the $d_k$-dimensional irreducible representation $R_k$ of $G$ (cf. Subsection \ref{s-gue200803yyd}), $\Delta_{d\mu} $ is the Laplacian induced by the probability Haar measure on $G$, $S_G$ is the scalar curvature of $G$ induced by the probability Haar measure on $G$ (see \eqref{e-gue200913yyd}), $R_e$ is the scalar curvature of $X$ in the direction of $G$ (see Definition~\ref{d-gue200914yyd}) and 
$V_{{\rm eff\,}}$ is given by \eqref{e-gue170108e}.
\end{theorem} 

\section{Preliminaries}\label{s:prelim} 

\subsection{Some standard notations }\label{s-gue170111w}

We shall use the following notations: $\mathbb N=\set{1,2,\ldots}$, 
$\mathbb N_0=\mathbb N\cup\set{0}$, $\mathbb R$ 
is the set of real numbers, $\overline{\mathbb R}_+:=\set{x\in\mathbb R;\, x\geq0}$. 
For a multi-index $\alpha=(\alpha_1,\ldots,\alpha_n)\in\mathbb N_0^n$,
we denote by $\abs{\alpha}=\alpha_1+\ldots+\alpha_n$ its norm.
For $m\in\mathbb N$, write $\alpha\in\set{1,\ldots,m}^n$ 
if $\alpha_j\in\set{1,\ldots,m}$, 
$j=1,\ldots,n$. $\alpha$ is strictly increasing 
if $\alpha_1<\alpha_2<\cdots<\alpha_n$. For $x=(x_1,\ldots,x_n)$, 
we write
\[
\begin{split}
&x^\alpha=x_1^{\alpha_1}\ldots x^{\alpha_n}_n,\\
& \pr_{x_j}=\frac{\pr}{\pr x_j}\,,\quad
\pr^\alpha_x=\pr^{\alpha_1}_{x_1}\ldots\pr^{\alpha_n}_{x_n}
=\frac{\pr^{\abs{\alpha}}}{\pr x^\alpha}\,.
\end{split}
\]
Let $z=(z_1,\ldots,z_n)$, $z_j=x_{2j-1}+ix_{2j}$, $j=1,\ldots,n$, 
be coordinates of $\mathbb C^n$. We write
\[
\begin{split}
&z^\alpha=z_1^{\alpha_1}\ldots z^{\alpha_n}_n\,,
\quad\ol z^\alpha=\ol z_1^{\alpha_1}\ldots\ol z^{\alpha_n}_n\,,\\
&\pr_{z_j}=\frac{\pr}{\pr z_j}=
\frac{1}{2}\Big(\frac{\pr}{\pr x_{2j-1}}-i\frac{\pr}{\pr x_{2j}}\Big)\,,
\quad\pr_{\ol z_j}=\frac{\pr}{\pr\ol z_j}
=\frac{1}{2}\Big(\frac{\pr}{\pr x_{2j-1}}+i\frac{\pr}{\pr x_{2j}}\Big),\\
&\pr^\alpha_z=\pr^{\alpha_1}_{z_1}\ldots\pr^{\alpha_n}_{z_n}
=\frac{\pr^{\abs{\alpha}}}{\pr z^\alpha}\,,\quad
\pr^\alpha_{\ol z}=\pr^{\alpha_1}_{\ol z_1}\ldots\pr^{\alpha_n}_{\ol z_n}
=\frac{\pr^{\abs{\alpha}}}{\pr\ol z^\alpha}\,.
\end{split}
\]
For $j, s\in\mathbb Z$, set $\delta_{j,s}=1$ if $j=s$, 
$\delta_{j,s}=0$ if $j\neq s$.

Let $M$ be a $m$-dimensional smooth orientable paracompact manifold. 
Let $TM$ and $T^*M$ denote the tangent bundle of $M$
and the cotangent bundle of $M$, respectively.
The complexified tangent bundle of $M$ and 
the complexified cotangent bundle of $M$ will be denoted by $\mathbb CTM:=\mathbb C\otimes_{\mathbb R}TM$
and $\mathbb CT^*M:=\mathbb C\otimes_{\mathbb R}T^*M$, respectively. We denote by 
$\langle\,\cdot\,,\cdot\,\rangle$ the pointwise
duality between $TM$ and $T^*M$.
We extend $\langle\,\cdot\,,\cdot\,\rangle$ bilinearly to 
$\mathbb C TM\times\mathbb C T^*M$.

Let $F$ be a smooth vector bundle over $M$. 
Let $D'(M,F)$ and $C^\infty(M,F)$ denote the spaces of distribution sections of $M$ with values in $F$ and smooth sections of $M$ with values in $F$ respectively.
We denote by $E'(M, F)$ the subspace of 
$D'(M, F)$ whose 
elements have compact support in $M$. Put $C^\infty_c(M,F):=C^\infty(M,F)\bigcap E'(M,F)$.

Let $W_1$ be an open set in $\Real^{N_1}$ and let $W_2$ be an open set in $\Real^{N_2}$. Let $E$ and $F$ be vector bundles over $W_1$ and $W_2$, respectively. 
An $m$-dependent continuous operator
$A_m: C^\infty_c(W_2,F)\To D'(W_1,E)$ is called $m$-negligible on $W_1\times W_2$
if, for $m$ large enough, $A_m$ is smoothing and, for any $K\Subset W_1\times W_2$, any
multi-indices $\alpha$, $\beta$ and any $N\in\mathbb N$, there exists $C_{K,\alpha,\beta,N}>0$
such that
\begin{equation*}\label{e-gue13628III}
\abs{\pr^\alpha_x\pr^\beta_yA_m(x, y)}\leq C_{K,\alpha,\beta,N}m^{-N}\:\: \text{on $K$, for all $m\gg 1$}.
\end{equation*}
In that case we write
\[A_m(x,y)=O(m^{-\infty})\:\:\text{on $W_1\times W_2$,}\]
or
\[A_m=O(m^{-\infty})\:\:\text{on $W_1\times W_2$.}\]
If $A_m, B_m: C^\infty_c(W_2, F)\To D'(W_1, E)$ are $m$-dependent continuous operators,
we write $A_m= B_m+O(m^{-\infty})$ on $W_1\times W_2$ or $A_m(x,y)=B_m(x,y)+O(m^{-\infty})$ on $W_1\times W_2$ if $A_m-B_m=O(m^{-\infty})$ on $W_1\times W_2$. 

We recall the definition of the semi-classical symbol spaces

\begin{definition} \label{d-gue140826}
Let $W$ be an open set in $\Real^N$. Let
\begin{gather*}
S(1;W):=\Big\{a\in C^\infty(W)\,|\, \forall\alpha\in\mathbb N^N_0:
\sup_{x\in W}\abs{\pr^\alpha a(x)}<\infty\Big\},\\
S^0_{{\rm loc\,}}(1;W):=\Big\{(a(\cdot,m))_{m\in\Real}\,|\, \forall\alpha\in\mathbb N^N_0,
\forall \chi\in C^\infty_c(W)\,:\:\sup_{m\in\Real, m\geq1}\sup_{x\in W}\abs{\pr^\alpha(\chi a(x,m))}<\infty\Big\}\,.
\end{gather*}
For $k\in\Real$, let
\[
S^k_{{\rm loc}}(1):=S^k_{{\rm loc}}(1;W)=\Big\{(a(\cdot,m))_{m\in\Real}\,|\,(m^{-k}a(\cdot,m))\in S^0_{{\rm loc\,}}(1;W)\Big\}\,.
\]
Hence $a(\cdot,m)\in S^k_{{\rm loc}}(1;W)$ if for every $\alpha\in\mathbb N^N_0$ and $\chi\in C^\infty_0(W)$, there
exists $C_\alpha>0$ independent of $m$, such that $\abs{\pr^\alpha (\chi a(\cdot,m))}\leq C_\alpha m^{k}$ holds on $W$.

Consider a sequence $a_j\in S^{k_j}_{{\rm loc\,}}(1)$, $j\in\N_0$, where $k_j\searrow-\infty$,
and let $a\in S^{k_0}_{{\rm loc\,}}(1)$. We say
\[
a(\cdot,m)\sim
\sum\limits^\infty_{j=0}a_j(\cdot,m)\:\:\text{in $S^{k_0}_{{\rm loc\,}}(1)$},
\]
if, for every
$\ell\in\N_0$, we have $a-\sum^{\ell}_{j=0}a_j\in S^{k_{\ell+1}}_{{\rm loc\,}}(1)$ .
For a given sequence $a_j$ as above, we can always find such an asymptotic sum
$a$, which is unique up to an element in
$S^{-\infty}_{{\rm loc\,}}(1)=S^{-\infty}_{{\rm loc\,}}(1;W):=\cap _kS^k_{{\rm loc\,}}(1)$.

Similarly, we can define $S^k_{{\rm loc\,}}(1;Y)$, $S^k_{{\rm loc\,}}(1;Y,E)$ in the standard way, where $Y$ is a smooth manifold and $E$ is a vector bundle over $Y$. 
\end{definition}

Let $b(m)$ be $m$-dependent function. We write $b(m)=O(m^{-\infty})$ if for every $N>0$, there is a constant $C_N>0$ such that 
$\abs{b(m)}\leq C_Nm^{-N}$, for every $m\gg1$. We write $b(m)\sim\sum^{+\infty}_{j=0}a_jm^{\ell-j}$, where $a_j\in\Complex$, $j=0,1,\ldots$, if for 
every $q\in\mathbb N$, we have $\abs{b(m)-\sum^q_{j=0}a_jm^{\ell-j}}\leq C_qm^{\ell-q-1}$, for all $m\gg1$, where $C_q>0$ is a constant independent of $m$.

\subsection{CR manifolds}\label{s-gue200807yyd}

Let $(X, T^{1,0}X)$ be a compact, connected and orientable CR manifold of dimension $2n+1$, $n\geq 1$, where $T^{1,0}X$ is a CR structure of $X$, that is, $T^{1,0}X$ is a subbundle of rank $n$ of the complexified tangent bundle $\mathbb{C}TX$, satisfying $T^{1,0}X\cap T^{0,1}X=\{0\}$, where $T^{0,1}X=\overline{T^{1,0}X}$, and $[\mathcal V,\mathcal V]\subset\mathcal V$, where $\mathcal V=C^\infty(X, T^{1,0}X)$. There is a unique subbundle $HX$ of $TX$ such that $\mathbb{C}HX=T^{1,0}X \oplus T^{0,1}X$, i.e. $HX$ is the real part of $T^{1,0}X \oplus T^{0,1}X$. Let $J: HX\To HX$ be the complex structure map given by $J(u+\ol u)=iu-i\ol u$, for every $u\in T^{1,0}X$. 
By complex linear extension of $J$ to $\mathbb{C}TX$, the $i$-eigenspace of $J$ is $T^{1,0}X \, = \, \left\{ V \in \mathbb{C}HX \, ;\, JV \, =  \,  \sqrt{-1}V  \right\}.$ We shall also write $(X, HX, J)$ to denote a compact CR manifold.  

From now on, we assume that $X$ admits a transversal and CR $S^1$-action $e^{i\theta}$. Let $T\in C^\infty(X,TX)$ be the global real vector field on $X$ induced by the $S^1$-action. Let $\omega_0\in C^\infty(X,T^*X)$ be the non-vanishing $1$-form on $X$ given by 
\begin{equation}\label{e-gue200803yydI}
\begin{split}
 &\mbox{$\langle\,\omega_0(x)\,,\,u\,\rangle=0$, for every $u\in H_xX$, for every $x\in X$},\\
 &\mbox{$\langle\,\omega_0\,,\,T\,\rangle=-1$ on $X$}.
 \end{split}
\end{equation}
The Levi form at $x\in X$ is the Hermitian  quadratic form on $T^{1,0}_xX$ given by 
\begin{equation}\label{e-gue200807yydh}
\mathcal{L}_x(U,\overline{V}):=-\frac{1}{2i}d\omega_0(U,\overline{V}),\ \ U, V\in T^{1,0}X. 
\end{equation}
In this paper, we assume that $X$ is strongly pseudoconvex, that is, $\mathcal{L}_x$ is positive definite at every point of $X$. 

The Levi form on $X$ induces a Hermitian metric $\langle\,\cdot\,|\,\cdot\,\rangle$ on $\Complex TX$ as \eqref{e-gue200726yyda}. 
 For $u \in \mathbb{C}TX$, we write $|u|^2 := \langle u | u \rangle$. Denote by $T^{*1,0}X$ and $T^{*0,1}X$ the dual bundles $T^{1,0}X$ and $T^{0,1}X$, respectively. They can be identified with subbundles of the complexified cotangent bundle $\mathbb{C}T^*X$. Define the vector bundle of $(0,q)$-forms by $T^{*0,q}X := \wedge^qT^{*0,1}X$. The Hermitian metric $\langle\,\cdot\,|\,\cdot\,\rangle$ on $\mathbb{C}TX$ induces, by duality, a Hermitian metric on $\mathbb{C}T^*X$ and also on the bundle 
 $\oplus^{2n+1}_{r=1}\Lambda^r(\Complex T^*X)$. We shall also denote all these induced metrics by $\langle\,\cdot\,|\,\cdot\,\rangle$. Note that we have the pointwise orthogonal decompositions:
\begin{equation*}
\begin{array}{c}
\mathbb{C}T^*X = T^{*1,0}X \oplus T^{*0,1}X \oplus \left\{ \lambda \omega_0: \lambda \in \mathbb{C} \right\}, \\
\mathbb{C}TX = T^{1,0}X \oplus T^{0,1}X \oplus \left\{ \lambda T: \lambda \in \mathbb{C} \right\}.
\end{array}
\end{equation*}

Let $D$ be an open set of $X$. Let $\Omega^{0,q}(D)$ denote the space of smooth sections of $T^{*0,q}X$ over $D$ and let $\Omega^{0,q}_c(D)$ be the subspace of $\Omega^{0,q}(D)$ whose elements have compact support in $D$. Let $(\,\cdot\,|\,\cdot\,)$ be the $L^2$ inner product on $\Omega^{0,q}(X)$ induced by $\langle\,\cdot\,|\,\cdot\,\rangle$. Let $L^2_{(0,q)}(X)$ be the completion of $\Omega^{0,q}(X)$ with respect to $(\,\cdot\,|\,\cdot\,)$. We extend $(\,\cdot\,|\,\cdot\,)$ to $L^2_{(0,q)}(X)$ in the standard way. We write $L^2(X):=L^2_{(0,0)}(X)$. 

We need the following classical result on local coordinates \cite{BRT85}.
\begin{theorem}\label{t-0529}
	For every point $p\in X$, we can find local coordinates $x=(x_1,\cdots,x_{2n+1})=(z,\theta)=(z_1,\cdots,z_{n},\theta), z_j=x_{2j-1}+ix_{2j},j=1,\cdots,n, x_{2n+1}=\theta$, defined in some small neighborhood $D=\{(z, \theta): \abs{z}<\delta, -\varepsilon_0<\theta<\varepsilon_0\}$ of $x_0$, $\delta>0$, $0<\varepsilon_0<\pi$, such that $(z(p),\theta(p))=(0,0)$ and
	\begin{equation}\label{e-can}
	\begin{split}
	&T=\frac{\partial}{\partial\theta}\\
	&Z_j=\frac{\partial}{\partial z_j}+i\frac{\partial\varphi}{\partial z_j}(z)\frac{\partial}{\partial\theta},j=1,\cdots,n
	\end{split}
	\end{equation}
	where $Z_j(x), j=1,\cdots, n$, form a basis of $T_x^{1,0}X$, for each $x\in D$ and $\varphi(z)\in C^\infty(D,\mathbb R)$ independent of $\theta$. We call $(D,(z,\theta),\varphi)$ BRT trivialization, $x=(z,\theta)$ canonical coordinates and  $\set{Z_j}^n_{j=1}$ BRT frames.
\end{theorem}

\subsection{Pseudohermitian geometry}
In this subsection we recall the definition of Tanaka-Webster curvature. For this moment, we do not assume that $(X, T^{1,0}X)$ admits a transversal CR $S^1$ action, that is, we only assume that $(X, T^{1,0}X)$ is a general oriented strongly pseudoconvex CR manifold of dimension $2n+1$, $n \ge 1$. Since $X$ is orientable, there is a $\theta_0 \in C^\infty(X, T^*X)$ which annihilates exactly $HX$. Any such $\theta_0$ is called a pseudohermitian structure on $X$. Then there is a unique vector field $\hat T \in C^\infty(X, TX)$ on $X$ such that
\[
\theta_0(\hat T) \equiv 1, \quad d\theta_0(\hat T, \cdot) \equiv 0.
\]
The following is well-known.
\begin{proposition}[\cite{Ta75}, Proposition 3.1]
With the notations above, there is a unique linear connection (Tanaka-Webster connection) denoted by $\nabla^{\theta_0} : C^\infty(X, TX) \to C^\infty(X, T^*X \otimes TX)$ satisfying the following conditions:
\begin{enumerate}
\item The contact structure $HX$ is parallel, i.e. $\nabla^{\theta_0}_UC^\infty(X, HX) \subset C^\infty(X, HX)$ for $U \in C^\infty(X, TX)$.
\item The tensor fields $\hat T$, $J$, $d\theta_0$ are parallel, i.e. $\nabla^{\theta_0}\hat T=0$, $\nabla^{\theta_0}J=0$, $\nabla^{\theta_0}d\theta_0 =0$.
\item The torsion $\tau$ of $\nabla^{\theta_0}$ satisfies: $\tau(U, V) = d\theta_0(U, V)\hat T$, $\tau(\hat T, JU) = - J\tau(\hat T, U)$, $U, V \in C^\infty(X, HX)$.
\end{enumerate}
\end{proposition}

Let $\{ Z_\alpha\}_{\alpha=1}^n$ be a local frame of $T^{1,0}X$ and let $\{ \theta^\alpha\}_{\alpha =1}^n$ be the dual frame of $\{ Z_\alpha\}_{\alpha=1}^n$. Write $Z_{\overline{\alpha}} = \overline{Z_\alpha}$, $\theta^{\overline{\alpha}} = \overline{\theta^\alpha}$. Write
\[
\nabla^{\theta_0}Z_\alpha = \omega_\alpha^\beta \otimes Z_\beta, \quad \nabla^{\theta_0}Z_{\overline{\alpha}} = \omega_{\overline{\alpha}}^{\overline{\beta}} \otimes Z_{\overline{\beta}}, \quad \nabla^{\theta_0}\hat T =0.
\]
we call $\omega_\alpha^\beta$ the connection form of Tanaka-Webster connection with respect to the frame $\{ Z_\alpha\}_{\alpha=1}^n$. We denote by $\Theta_\alpha^\beta$ the Tanaka-Webster curvature form. We have $\Theta_\alpha^\beta = d\omega_\alpha^\beta - \omega_\alpha^\gamma \wedge \omega_\gamma^\beta$. It is easy to check that
\[
\Theta_\alpha^\beta = R_\alpha{}^\beta{}_{j\overline{k}} \theta^j \wedge \theta^{\overline{k}} + A_\alpha{}^\beta{}_{jk} \theta^j \wedge \theta^k + B_\alpha{}^\beta{}_{jk} \theta^{\overline{j}} \wedge \theta^{\overline{k}} + C \wedge \theta_0, \quad \text{$C$ is a one-form}. 
\]
We call $R_\alpha{}^\beta{}_{j\overline{k}}$ the pseudohermitian curvature tensor and its trace
\[
R_{\alpha\overline{k}} := \sum_{j=1}^n R_\alpha{}^j{}_{j\overline{k}} 
\]
is called pseudohermitian Ricci tensor. Write $d\theta_0 = ig_{\alpha\overline{\beta}} \theta^\alpha \wedge \theta^{\overline{\beta}}$. Let $\{ g^{\overline{\sigma}\beta} \}$ be the inverse matrix $\{ g_{\alpha \overline{\beta}} \}$. The Tanaka-Webster scalar curvature $R$ with respect to the pseudohermitian structure $\theta_0$ is given by 
\begin{equation}\label{e-gue200825twsc}
R = g^{\overline{k}\alpha}R_{\alpha \overline{k}}.
\end{equation}

In this paper, we will take $\theta_0=-d\omega_0$ and hence $\hat T=T$. 

\subsection{$G$-equivariant Szeg\H{o} kernels}\label{s-gue200803yyd} 

From now on, we assume that $X$ admits a $d$-dimensional connected compact Lie group $G$ action and Assumption~\ref{a-gue170123I}, Assumption~\ref{a-gue170123II}  hold. We first introduce some notations in representation theory. 

We recall that a representation of the group $G$ is a group homomorphism $\rho: G\To GL(\Complex^d)$ for some $d\in\mathbb N$. The representation represents the elements of the group as $d\times d$ complex square matrices so that multiplication commutes with $\rho$. The number $d$ is the dimension of the representation $\rho$. 
A representation $\rho$ is unitary if each $\rho(g)$, $g\in G$, is an unitary matrix. A representation $\rho$ is reducible if we have a splitting $\Complex^d=V_1\oplus V_2$ so that $\rho(g)V_j =V_j$ for all $g\in G$, for both $j=1,2$ and $0<{\rm dim\,}V_1<d$, where $V_1$ and $V_2$ are vector subspaces of $\Complex^d$. If $\rho$ is not reducible, it is called irreducible. Two representations $\rho_1$ and $\rho_2$ are equivalent if they have the same dimension and there is an invertible matrix $A$ such that $\rho_1(g)=A\rho_2(g)A^{-1}$  for all $g\in G$. To understand all representations of the group $G$, it often suffices to study the irreducible unitary representations. Let 
\begin{equation*}
	R=\{R_1, R_2,...\}
\end{equation*}
be the collection of all irreducible unitary representations of $G$, where each $R_k$
comes from exactly only one equivalent class. For each $R_k$, we  write $R_k$ as a matrix $\left(R_{k,j,\ell}\right)^{d_k}_{j,\ell=1}$, where $d_k$ is the dimension of $R_k$. 
Denote by $\chi_k(g):=\text{Tr}R_k(g)$ the trace of the matrix $R_k(g)$ (the character of $R_k$). 

Fix $m\in\mathbb Z$, let 
\begin{equation}\label{e-gue200807yydi}
S_m: L^2(X)\To{\rm Ker\,}\ddbar_b\bigcap L^2_m(X)
\end{equation}
be the orthogonal projection with respect to $(\,\cdot\,|\,\cdot\,)$. For each $R_k$, the $m$-th Fourier component of the $R_k$-th $G$-equivariant Szeg\H{o} projection is the orthogonal projection 
\begin{equation}\label{e-gue200807yydj}
S_{k,m}: L^2(X)\To{\rm Ker\,}\ddbar_b\bigcap L^2_m(X)_k
\end{equation}
with respect to $(\,\cdot\,|\,\cdot\,)$. Let $S_m(x,y)$, $S_{k,m}(x,y)\in C^\infty(X\times X)$ be the distribution kernels of $S_m$, $S_{k,m}$ respectively. 
Fix a Haar measure $d\mu(g)$ on $G$ so that $\int_Gd\mu(g)=1$. It is not difficult to see that 
\begin{equation}\label{e-gue200803ycdh}
S_{k,m}(x,y)=d_k\int_G S_m(g\circ x, y)\ol{\chi_k(g)}d\mu(g). 
\end{equation}

\section{Local expression for coefficients of lower order terms}\label{s-gue200809yyd}

We use the same notations and assumptions above.
Note that $X$ is strongly pseudoconvex. In this section, we compute the coefficients of the first two lower order terms of the asymptotic expansion \eqref{e-gue200731yyd}. We will first recall the asymptotic expansion result for $S_m$.  We introduce some geometric objects in Theorem~\ref{t-1910} below. 
For $u\in\Lambda^r(\Complex T^*X)$, we denote $\abs{u}^2:=\langle\,u\,|\,u\,\rangle$.  Let $x=(z,\theta)$ be canonical coordinates on an open set $D\subset X$. Until further notice, we will work with the canonical coordinates $x=(z,\theta)$. 
Let $Z_1\in C^\infty(D,T^{1,0}X),\ldots,Z_n\in C^\infty(D,T^{1,0}X)$ be as in Theorem \ref{t-0529} and let $e_1\in C^\infty(D,T^{*1,0}X),\ldots,e_n\in C^\infty(D,T^{*1,0}X)$ be the dual frames. The CR rigid Laplacian with respect to $\langle\,\cdot\,|\,\cdot\,\rangle$ is given by
\begin{equation} \label{s1-e9m}
\triangle_{\mathcal{L}}=(-2)\sum^n_{j,\ell=1}\langle\,e_j\,|\,e_\ell\,\rangle Z_j\ol{Z_\ell}.
\end{equation}
It is easy to check that $\triangle_{\mathcal{L}}$ is globally defined.
Let
\begin{equation*}\label{e-gue160531m}
\frac{1}{n!}\Bigr((-\frac{1}{2\pi}d\omega_0)^n\wedge(-\omega_0)\Bigr)(x)=a(x)dx_1\cdots dx_{2n+1}\ \ \mbox{on $D$},
\end{equation*}
where $a(x)\in C^\infty(D)$.
The rigid scalar curvature $S_{\mathcal{L}}$ is given by
\begin{equation}\label{e-gue160531Im}
S_{\mathcal{L}}(x):=\triangle_{\mathcal{L}}(\log a(x)).
\end{equation}
It is easy to see that $S_{\mathcal{L}}(x)$ is well-defined, $S_{\mathcal{L}}(x)\in C^\infty(X)$, $TS_{\mathcal{L}}(x)=0$, cf. \cite{HHL18}. 
It is shown in \cite[Theorem 3.5]{HHL18} that $S_{\mathcal{L}}(x)=4 R(x)$, where $R$ (cf. \eqref{e-gue200825twsc}) is the Tanaka-Webster scalar curvature with respect to the pseudohermitian structure $-\omega_0$. 


We recall the following \cite{HHL18} 

\begin{theorem}\label{t-1910}
Recall that we work with the assumptions that $X$ is a compact connected orientable strongly pseudoconvex CR manifold of dimension $2n+1$, which admits a transversal CR $S^1$ action. For every $m\in\mathbb Z$, denote by $S_m(x,y)$ the distribution kernel of the orthogonal projection $S_m$ (cf. \eqref{e-gue200807yydi}). 
Fix $p\in X$ and assume that the $S^1$ action is free near $p$. 
Let $x=(z,\theta)$ be canonical coordinates on an open set $D\subset X$ of $p$ with $(z(p),\theta(p))=(0,0)$. Then near $(0,0)$, as $m\To+\infty$, 
\begin{equation}\label{e-gue200809yyd}
\begin{split}
&S_m(x,y)=e^{im\Phi(x,y)}a(x,y,m)+O(m^{-\infty}),\\
&\mbox{$a(x,y,m)\sim\sum^{+\infty}_{j=0}m^{n-j}a_j(x,y)$ in $S^n_{{\rm loc\,}}(1;D\times D)$, $a_j(x,y)\in C^\infty(X\times X)$, $j=0,1,\ldots$},
\end{split}
\end{equation}
where $\Phi(x,y)=x_{2n+1}-y_{2n+1}+\hat\Phi(z,w)$,
\begin{equation}\label{e-gue200809yyda}
\hat\Phi(z,w)=i(\varphi(z)+\varphi(w))-2i\sum_{|\alpha|+|\beta|\leq N}\frac{\pr^{|\alpha|+|\beta|}\varphi}{\pr z^\alpha\pr \bar z^\beta}(0)\frac{z^\alpha}{\alpha!}\frac{\bar w^\beta}{\beta!}+O(|(z,w)|^{N+1}), \ \ \mbox{for every $N\in\mathbb N$},
\end{equation}
where $\varphi(z,w)\in C^\infty(D\times D)$ is as in \eqref{e-can}. 
Moreover, for $a_0(x,y)$, $a_1(x,y)$ in \eqref{e-gue200809yyd}, we have 
\begin{equation}\label{e-1912}
\begin{split}
a_0(x,y)&=a_0(x):=a_0(x,x)=\frac{1}{2\pi^{n+1}}, \\
a_1(x):&=a_1(x,x)=\frac{1}{16\pi^{n+1}}S_{\mathcal{L}}(x) = \frac{1}{4\pi^{n+1}}R(x).
\end{split}
\end{equation}
Note that in \eqref{e-1912}, $a_0(x,y)$ is a constant function. 
\end{theorem}
Note that the Levi metric in \eqref{e-gue200726yyda} is slightly different from that in \cite[(3.8)]{HHL18}.

Until further notice, we fix an irreducible representation $R_k$ and we fix $p\in\mu^{-1}(0)$. As $m\To+\infty$, we have 
\begin{equation}\label{e-1913}
\begin{split}
S_{k,m}(p):&=S_{k,m}(p,p)=d_k\int_G S_m(g\circ p,p)\ol{\chi_k(g)}d\mu(g)\\
&= d_k\int_G e^{im\Phi(g\circ p,p)}a(g\circ p,p,m)\ol{\chi_k(g)}d\mu(g)+O(m^{-\infty}).
\end{split}
\end{equation}
Put $Y_p=\set{g\circ p;\, g\in G}$, then $Y_p$ is a $d$-dimensional submanifold of $X$. The $G$-invariant Hermitian metric $\langle\,\cdot\,|\,\cdot\,\rangle$ induces a volume form $dv_{Y_p}$ on $Y_p$. Set
\begin{equation}\label{e-gue170108e}
V_{{\rm eff\,}}(p):=\int_{Y_p}dv_{Y_p}.
\end{equation}
For $f(g)\in C^\infty(G)$, let $\hat f(g\circ p):=f(g)$, $\forall g\in G$. Then, $\hat f\in C^\infty(Y_p)$. Let $d\hat\mu$ be the measure on $G$ given by $\int_Gfd\hat\mu:=\int_{Y_p}\hat fdv_{Y_p}$, for all $f\in C^\infty(G)$. It is not difficult to see that $d\hat\mu$ is a Haar measure and 
\begin{equation}\label{e-gue170108}
\int_Gd\hat\mu=V_{{\rm eff\,}}(p). 
\end{equation}

Let $e_0$ denote the identity element of $G$. Let $d\mu$ be the Haar measure  on $G$ such that $\int_Gd\mu=1$ and let $\langle\,\cdot\,|\,\cdot\,\rangle_{d\mu}$ be the Hermitian metric on $TG$ such that 
$\langle\,\cdot\,|\,\cdot\,\rangle_{d\mu}$ induces the Haar measure $d\mu$ on $G$. From now on, we fix $p\in\mu^{-1}(0)$. We need ( see \cite[Theorem 3.6]{HH})

\begin{theorem}\label{t-gue161202}
There exist local coordinates $y'=(y_1,\ldots,y_d)$ of $G$ defined in  a neighborhood $W$ of $e_0$ with $y'(e_0)=(0,\ldots,0)$
such that
\begin{equation}\label{e-gue200810yyd}
\begin{split}
&\mbox{$\langle\,\frac{\pr}{\pr y_j}\,|\,\frac{\pr}{\pr y_\ell}\,\rangle_{d\mu}=2\Bigr(V_{{\rm eff\,}}(p)\Bigr)^{-\frac{2}{d}}\delta_{j,\ell}+O(\abs{y}^2)$, for every $j, \ell=1,\ldots,d$},\\
&\mbox{$(y_1,\ldots,y_d)\circ (0,\ldots,0)=(y_1,\ldots,y_d,0,\ldots,0)$, for every $(y_1,\ldots,y_d)\in W$}, 
\end{split}
\end{equation}
and we can find local coordinates $y=(y_1,\ldots,y_d, y_{d+1},\ldots,y_{2n+1})$ of $X$ defined in a neighborhood $U=U_1\times U_2$ of $p$ with $0\leftrightarrow p$, where $U_1\subset W$ is an open neighborhood of $0\in\Real^d$,  $U_2\subset\Real^{2n+1-d}$ is an open neighborhood of $0\in\Real^{2n+1-d} $, such that
	\begin{equation}\label{e-gue200810yydI}
	\begin{split}
	&T=\frac{\pr}{\pr x_{2n+1}},\\
	&\underline{\mathfrak{g}}={\rm span\,}\set{\frac{\pr}{\pr y_1},\ldots,\frac{\pr}{\pr y_d}},\\
	&T^{1,0}_pX={\rm span\,}\set{Z_1,\ldots,Z_n},\\
	&Z_j=\frac{1}{2}(\frac{\pr}{\pr y_j}-i\frac{\pr}{\pr y_{d+j}})(p),\ \ j=1,\ldots,d,\\
	&Z_j=\frac{1}{2}(\frac{\pr}{\pr y_{2j-1}}-i\frac{\pr}{\pr y_{2j}})(p),\ \ j=d+1,\ldots,n,\\
	&\mathcal{L}_p(Z_j, \ol Z_\ell)=\delta_{j,\ell},\ \ j,\ell=1,2,\ldots,n.
	\end{split}
	\end{equation}		
\end{theorem}  

We will also identify $\frac{\pr}{\pr y_j}$, $j=1,\ldots,d$, as vector fields on $X$. 
From \eqref{e-gue170108}, it is straightforward to see that 
\begin{equation}\label{e-gue200811yyd}
\mbox{$\langle\,\frac{\pr}{\pr y_j}\,|\,\frac{\pr}{\pr y_\ell}\,\rangle=\Bigr(V_{{\rm eff\,}}(p)\Bigr)^{\frac{2}{d}}\langle\,\frac{\pr}{\pr y_j}\,|\,\frac{\pr}{\pr y_\ell}\,\rangle_{d\mu}$, $j, \ell=1,\ldots,d$}. 
\end{equation}

 Let $x=(z,\theta)$ be canonical coordinates on an open neighborhood $D\subset X$ of $p$ with $(z(p),\theta(p))=(0,0)$. 
Let $U_1$ and $U_2$ be open sets in Theorem~\ref{t-gue161202}. We take $U_1$ and $U_2$ small enough so that 
$U_1\times U_2\subset D$ in a BRT chart $D$ of $p$. As before, let $y'=(y_1,\ldots,y_d)$. From now on, we take $x=(z,\theta)$ so that 
\begin{equation}\label{e-gue200810yydII}
\begin{split}
&\frac{\pr}{\pr y_j}=\frac{\pr}{\pr x_{2j-1}}+O(\abs{y'}) \  \mbox{at $(y_1,\ldots,y_d,0,\ldots,0)$, $j=1,\ldots,d$},\\
&\varphi(z)=\sum^n_{j=1}\abs{z_j}^2+O(\abs{z}^4), \\
&\frac{\pr^4\varphi}{\pr z_\alpha\pr z_\beta\pr z_\gamma\pr z_\delta}(0)=0,\ \ \mbox{for every $\alpha, \beta, \gamma, \delta=1,\ldots,n$},\\
&\frac{\pr^4\varphi}{\pr z_\alpha\pr z_\beta\pr z_\gamma\pr\ol z_\delta}(0)=0,\ \ \mbox{for every $\alpha, \beta, \gamma, \delta=1,\ldots,n$},
\end{split}
\end{equation}
where $\varphi(z)\in C^\infty(D,\Real)$ is as in \eqref{e-can}. 

From now on, we will work with $y$-coordinates and we will identify $y'=(y_1,\ldots,y_d)$ as local coordinates of $G$ defined near $e_0$. 
In $y'$-coordinates, we write $d\mu(g)=V(y')dy'$, $V(y')\in C^\infty(U_1)$. From the first property of \eqref{e-gue200810yyd}, we see that 
\begin{equation}\label{e-gue200811yyda}
\begin{split}
&V(0)=2^{\frac{d}{2}}\Bigr(V_{{\rm eff\,}}(p)\Bigr)^{-1},\\
&(\frac{\pr}{\pr y_j}V)(0)=0,\ \  j=1,...,d.
\end{split}
\end{equation}

Let $\tau(y')\in C^\infty_0(U_1)$, $\tau=1$ near $y'=0$. We have
\begin{equation}\label{e-1915}
S_{k,m}(p)=d_k\int e^{im\Phi(y',0)}a(y',0,m)\ol{\chi_k(y')}V(y')\tau(y')dy'+O(m^{-\infty}).
\end{equation}

Let us recall H\"ormander's stationary phase formula \cite{Hor03}
\begin{theorem}\label{t-1911}
Let $D$ be an open subset in $\mathbb{R}^N$ and $F$ be a smooth function on $D$.
If $\operatorname{Im} F\geq 0$, $\operatorname{Im} F(0)=0, F'(0)=0$, $\det F''(0)\neq 0$, then
\begin{equation*}\label{e-1916}
\int_{\mathbb{R}^N}e^{imF(x)}u(x)dx\sim e^{imF(0)}\left(\det \frac{mF''(0)}{2\pi i} \right)^{-\frac{1}{2}}\sum_{j=0}^\infty m^{-j}L_j u.
\end{equation*}
Here 
\begin{equation}\label{e-1917}
L_j u=\sum_{\nu-\mu=j}\sum_{2\nu\geq 3\mu}i^{-j}2^{-\nu}\langle F''(0)^{-1}D,D\rangle^{\nu}\left(\frac{h^{\mu}u}{\nu!\mu!} \right)(0),
\end{equation}
where $D=(-i\pr_{x_1},...,-i\pr_{x_N})^T$, $h(x)=F(x)-F(0)-\frac{1}{2}\langle F''(0)x,x\rangle$.
\end{theorem}

In our case, $F(y')=\Phi(y',0)$, $u(y')=a(y',0,m)\ol{\chi_k(y')}V(y')\tau(y')$.
Since $p\in\mu^{-1}(0)$, we have 
\begin{equation*}
\frac{\pr}{\pr y_j}	\Phi(p,p)=\langle\,\omega_0(p)\,,\, \frac{\pr}{\pr y_j}\,\rangle=0,
\end{equation*}
for $j=1,...,d$. 
Applying Theorem \ref{t-1911}, we have
\begin{equation}\label{e-1918}
\begin{split}
S_{k,m}(p)&\sim c_0d_k m^{-\frac{d}{2}}\sum_{j=0}^\infty m^{-j}L_j u\\
&\sim m^{n-\frac{d}{2}}b_{0,k}(p)+m^{n-\frac{d}{2}-1}b_{1,k}(p)+\cdots, 
\end{split}
\end{equation}
for some constant $c_0$. 
Our goal is to compute the coefficients $b_{0,k}(p), b_{1,k}(p)$. From \eqref{e-gue200809yyda} and \eqref{e-gue200810yydII}, 
we see that 
\begin{equation}\label{e-1919}
F(y')=\Phi(y',0)=i\sum_{j=1}^d y_j^2+O(|y'|^3),
\end{equation}
which satisfies the conditions in Theorem \ref{t-1911}.
Moreover, 
\begin{equation}\label{e-gue200811yydh}
	\begin{split}
	&\Big(\det \frac{mF''(0)}{2\pi i} \Big)^{-\frac{1}{2}}=\pi^{\frac{d}{2}}m^{-\frac{d}{2}},\\\
	&\langle F''(0)^{-1}D,D\rangle=\frac{i}{2}\sum_{j=1}^d\frac{\pr^2}{\pr y_j^2}=:\frac{i}{2}\Delta.
	\end{split}
\end{equation}
So $c_0=\pi^{\frac{d}{2}}$ and 
\begin{equation}\label{e-gue200812yyd}
S_{k,m}(p)\sim \pi^{\frac{d}{2}}d_k m^{-\frac{d}{2}}\sum_{j=0}^\infty m^{-j}L_j u.
\end{equation} 

At $(y_1,\ldots,y_d,0,\ldots,0)$, write 
\begin{equation}\label{e-gue200817yyd}
\begin{split}
&\frac{\pr}{\pr y_j}=\sum^{n}_{\ell=1}\Bigr(\hat a_{j,\ell}Z_{\ell}+\ol{\hat a_{j,\ell}}\ol {Z_{\ell}}\Bigr),\ \ j=1,\ldots,d,\\
&J(\frac{\pr}{\pr y_j})=\sum^{n}_{\ell=1}\Bigr(\hat a_{j,\ell}iZ_\ell+\ol{\hat a_{j,\ell}}(-i)\ol {Z_{\ell}}\Bigr),\ \ j=1,\ldots,d,
\end{split}\end{equation}
where $Z_\ell\in T^{1,0}X$ is as in \eqref{e-can}, $\ell=1,\ldots,n$, $\hat a_{j,\ell}\in\mathbb C$, $j=1,\ldots,d$, $\ell=1,\ldots,n$. We rewrite \eqref{e-gue200817yyd}: 
\begin{equation}\label{e-gue200812yydI}
\begin{split}
&	\frac{\pr}{\pr y_j}=\frac{\pr}{\pr x_{2j-1}}+\sum^n_{\ell=1} a_{j,\ell}\frac{\pr}{\pr z_\ell}+
	\sum_{\ell=1}^n \bar a_{j,\ell}\frac{\pr}{\pr \bar z_\ell}+b_j \frac{\pr}{\pr x_{2n+1}},\ \ j=1,\ldots,d,\\
	&	J(\frac{\pr}{\pr y_j})=\frac{\pr}{\pr y_{2j-1}}+i\sum^n_{\ell=1} a_{j,\ell}\frac{\pr}{\pr z_\ell}-i
	\sum^n_{\ell=1} \bar a_{j,\ell}\frac{\pr}{\pr \bar z_\ell}+c_j \frac{\pr}{\pr x_{2n+1}},\ \ j=1,\ldots,d,
	\end{split}
\end{equation}
where $\hat a_{j,\ell}=\delta_{j,\ell}+a_{j,\ell}$, $b_j, c_j\in\mathbb R$, $j=1,\ldots,d$, $\ell=1,\ldots,n$. 
We need

\begin{proposition}\label{p-210}
With the notations used above, we have  $a_{j,\ell}=O(|y'|), b_j=O(|y'|^2)$, $j=1,\ldots,d$, $\ell=1,\ldots,n$.
\end{proposition}

\begin{proof}
We only need to show $b_j=O(|y'|^2)$. At $(y_1,\ldots,y_d,0,\ldots,0)\in\mu^{-1}(0)$, we have 
\begin{equation}\label{e-gue200812yydII}
0=\omega_0(\frac{\pr}{\pr y_j})=\omega_0(\frac{\pr}{\pr x_{2j-1}})+\sum_{\ell=1}^n a_{j,l}\omega_0(\frac{\pr}{\pr z_\ell})+
	\sum_{\ell=1}^n \bar a_{j,\ell}\omega_0(\frac{\pr}{\pr \bar z_\ell})-b_j, 
\end{equation}
 $j=1,\ldots,d$. Note that at $(y_1,\ldots,y_d,0,\ldots,0)$, we have
 \[\begin{split}
 &\omega_0(\frac{\pr}{\pr x_{2j-1}})=O(\abs{y'}^2),\ \ j=1,\ldots,d,\\
 &\sum_{\ell=1}^n a_{j,\ell}\omega_0(\frac{\pr}{\pr z_\ell})+
	\sum_{\ell=1}^n \bar a_{j,\ell}\omega_0(\frac{\pr}{\pr \bar z_\ell})=O(\abs{y'}^2).\end{split}\]
From this observation and \eqref{e-gue200812yydII}, the proposition follows. 
\end{proof}

We claim that the third order term of $h$ vanishes.
\begin{proposition}\label{p-1910}
Let $h$ be the function defined in Theorem \ref{t-1911} with $F(y')=\Phi(y',0)$, then
$h=O(|y'|^4)$.
\end{proposition}
\begin{proof}
Since
\begin{equation*}\label{e-210}
F(y')=\Phi(x,0)=\Phi(x(y'),0)=i\varphi(y')+O(|y'|^N),\ \ \mbox{for every $N\in\mathbb N$}, 
\end{equation*}
we have
\begin{equation*}
	\frac{\pr^3 h}{\pr y_j\pr y_\ell\pr y_s}(0)=i\frac{\pr^3 \varphi}{\pr y_j\pr y_\ell\pr y_s}(0),
\end{equation*}
for all $j,\ell,s=1,...,d$.
By the definition of Levi metric, we have
\begin{equation*}
	\begin{split}
		\langle \frac{\pr}{\pr y_j}| \frac{\pr}{\pr y_\ell}\rangle=&
		\langle  \frac{\pr}{\pr z_j}+\frac{\pr}{\pr \bar z_j}+ \sum_{\alpha=1}^n a_{j,\alpha}\frac{\pr}{\pr z_\alpha}+
		\sum_{\alpha=1}^n \bar a_{j,\alpha}\frac{\pr}{\pr \bar z_\alpha}+b_j \frac{\pr}{\pr x_{2n+1}} | \\
		&   \frac{\pr}{\pr z_\ell}+\frac{\pr}{\pr \bar z_\ell}+ \sum_{\beta=1}^n a_{\ell,\beta}\frac{\pr}{\pr z_\beta}+
		\sum_{\beta=1}^n \bar a_{\ell,\beta}\frac{\pr}{\pr \bar z_\beta}+b_\ell \frac{\pr}{\pr x_{2n+1}}  \rangle\\
		=&\frac{\pr^2\varphi}{\pr z_j\pr \bar z_\ell}+\frac{\pr^2\varphi}{\pr \bar z_j \pr z_\ell}+
		\sum_{\beta=1}^n \bar a_{\ell,\beta}\frac{\pr^2\varphi}{\pr z_j\pr \bar z_\beta}+
		\sum_{\beta=1}^n a_{\ell,\beta}\frac{\pr^2\varphi}{\pr \bar z_j\pr z_\beta}\\
		&+\sum_{\alpha=1}^n a_{j,\alpha}\frac{\pr^2\varphi}{\pr z_\alpha\pr \bar z_\ell}+
		\sum_{\alpha=1}^n \bar a_{j,\alpha}\frac{\pr^2\varphi}{\pr \bar z_\alpha\pr z_\ell}+O(|y'|^2)=2\delta_{j,\ell}+O(|y'|^2).
	\end{split}
\end{equation*}
Then
\begin{equation*}
	2\delta_{j,\ell}=\frac{\pr^2\varphi}{\pr z_j\pr \bar z_\ell}+\frac{\pr^2\varphi}{\pr \bar z_j \pr z_\ell}+(a_{\ell,j}+\bar a_{\ell,j})+(a_{j,\ell}+\bar a_{j,\ell}) \mod |y'|^2.
\end{equation*}
So 
\begin{equation}\label{e-211}
\frac{\pr}{\pr y_s}(a_{\ell,j}+\bar a_{\ell,j}+a_{j,\ell}+\bar a_{j,\ell})(0)=0, 
\end{equation}
for $j,\ell,s=1,...,d$.
Note that 
\begin{equation*}
	\frac{\pr\varphi}{\pr y_j}=
	 \frac{\pr\varphi}{\pr z_j}+\frac{\pr\varphi}{\pr \bar z_j}+ \sum_{\alpha=1}^n a_{j,\alpha}\frac{\pr\varphi}{\pr z_\alpha}+
	\sum_{\alpha=1}^n \bar a_{j,\alpha}\frac{\pr\varphi}{\pr \bar z_\alpha}.
\end{equation*}
Then
\begin{equation*}
	\begin{split}
	\frac{\pr^2\varphi}{\pr y_\ell\pr y_j}=&\big(\frac{\pr}{\pr z_\ell}+\frac{\pr}{\pr \bar z_\ell}+ \sum_{\beta=1}^n a_{\ell,\beta}\frac{\pr}{\pr z_\beta}+
	\sum_{\beta=1}^n \bar a_{\ell,\beta}\frac{\pr}{\pr \bar z_\beta}\big)\\
	&\big(\frac{\pr\varphi}{\pr z_j}+\frac{\pr\varphi}{\pr \bar z_j}+ \sum_{\alpha=1}^n a_{j,\alpha}\frac{\pr\varphi}{\pr z_\alpha}+
	\sum_{\alpha=1}^n \bar a_{j,\alpha}\frac{\pr\varphi}{\pr \bar z_\alpha}\big)+\\
	&\sum_{\alpha=1}^n\big(b_\ell \frac{\pr a_{j,\alpha}}{\pr x_{2n+1}}\frac{\pr\varphi}{\pr z_\alpha} +b_\ell \frac{\pr \bar a_{j,\alpha}}{\pr x_{2n+1}}\frac{\pr\varphi}{\pr\bar z_\alpha}\big)\\
	=&\big(\frac{\pr^2\varphi}{\pr z_\ell\pr z_j}+\frac{\pr^2\varphi}{\pr z_\ell\pr \bar z_j} +\frac{\pr^2\varphi}{\pr \bar z_\ell\pr z_j}+\frac{\pr^2\varphi}{\pr \bar z_\ell\pr \bar z_j} \big)+\\
	&\sum_{\alpha=1}^n\big(a_{j,\alpha}\frac{\pr^2\varphi}{\pr z_\ell\pr z_\alpha}+\bar a_{j,\alpha}\frac{\pr^2\varphi}{\pr z_\ell\pr \bar z_\alpha}+a_{j,\alpha}\frac{\pr^2\varphi}{\pr \bar z_\ell\pr z_\alpha}+\bar a_{j,\alpha}\frac{\pr^2\varphi}{\pr \bar z_\ell\pr \bar z_\alpha}  \big)+\\
	&\sum_{\beta=1}^n\big(a_{\ell,\beta}\frac{\pr^2\varphi}{\pr z_j\pr z_\beta}+\bar a_{\ell,\beta}\frac{\pr^2\varphi}{\pr z_j\pr \bar z_\beta}+a_{\ell,\beta}\frac{\pr^2\varphi}{\pr\bar z_j\pr z_\beta}+\bar a_{\ell,\beta}\frac{\pr^2\varphi}{\pr \bar z_j\pr \bar z_\beta}  \big)+\\
	&\sum_{\alpha=1}^n\big(\frac{\pr a_{j,\alpha}}{\pr y_\ell} \frac{\pr\varphi}{\pr z_\alpha}+
	\frac{\pr \bar a_{j,\alpha}}{\pr y_\ell} \frac{\pr\varphi}{\pr \bar z_\alpha} \big)+O(|y'|^2).
	\end{split}
\end{equation*}
Hence
\begin{equation}\label{e-212}
\frac{\pr}{\pr y_s}(\frac{\pr^2\varphi}{\pr y_\ell\pr y_j})(0)=\frac{\pr}{\pr y_s}(a_{\ell,j}+\bar a_{\ell,j}+a_{j,\ell}+\bar a_{j,\ell})(0)+\frac{\pr}{\pr y_j}(a_{\ell,s}+\bar a_{\ell,s})(0).
\end{equation}
Similarly we have
\begin{equation}\label{e-213}
\frac{\pr}{\pr y_\ell}(\frac{\pr^2\varphi}{\pr y_s\pr y_j})(0)=\frac{\pr}{\pr y_\ell}(a_{s,j}+\bar a_{s,j}+a_{j,s}+\bar a_{j,s})(0)+\frac{\pr}{\pr y_j}(a_{s,\ell}+\bar a_{s,\ell})(0).
\end{equation}
By \eqref{e-212}, \eqref{e-213} and \eqref{e-211}, we have
\begin{equation*}\label{e-214}
\begin{split}
&2\frac{\pr^3 \varphi}{\pr y_j\pr y_\ell\pr y_s}(0)=\frac{\pr}{\pr y_s}(a_{\ell,j}+\bar a_{\ell,j}+a_{j,\ell}+
\bar a_{j,\ell})(0)+\\
&\frac{\pr}{\pr y_\ell}(a_{s,j}+\bar a_{s,j}+a_{j,s}+\bar a_{j,s})(0)+\frac{\pr}{\pr y_j}(a_{s,\ell}+\bar a_{s,\ell}+a_{\ell,s}+\bar a_{\ell,s})(0)=0.
\end{split}
\end{equation*}
The proof is completed.
\end{proof}

We can actually show 

\begin{proposition}\label{p-gue200817yyd}
At $(y_1,\ldots,y_d,0,\ldots,0)$, we have $\frac{\pr a_{j,\ell}}{\pr y_s}(0)=0$, $j,s=1,\ldots,d$, $\ell=1,\ldots,n$, where 
$a_{j,\ell}$, $j=1,\ldots,d$, $\ell=1,\ldots,n$, are as in \eqref{e-gue200812yydI}. 
\end{proposition}

\begin{proof}
From the proof of Proposition~\ref{p-1910} (see \eqref{e-212}), we see that 
\begin{equation}\label{e-gue200817yyds}
\frac{\pr}{\pr y_s}(a_{j,\ell}+\bar a_{j,\ell})(0)=0,\ \ j,s=1,\ldots,d, \ell=1,\ldots,n.
\end{equation}
Note that $\langle\,J(\frac{\pr}{\pr y_j})\,|\,J(\frac{\pr}{\pr y_\ell})\,\rangle=2\delta_{j,\ell}+O(\abs{y'}^2)$,  $j,\ell=1,\ldots,d$.
From this observation and \eqref{e-gue200812yydI}, we can repeat the proof of Proposition~\ref{p-1910} and deduce that 
\begin{equation}\label{e-gue200817yydu}
\frac{\pr}{\pr y_s}(a_{j,\ell}-\bar a_{j,\ell})(0)=0,\ \ j, s=1,\ldots,d, \ell=1,\ldots,n.
\end{equation}
From \eqref{e-gue200817yyds} and \eqref{e-gue200817yydu}, the proposition follows. 
\end{proof}

It follows from Proposition \ref{p-1910} that we only consider the cases when $2\nu\geq 4\mu$ for $L_j u$ in \eqref{e-1917}. 
Now we compute $L_0 u$ and $L_1 u$.

Let $j=0$, then $\mu=\nu=0$. From \eqref{e-1912} and \eqref{e-gue200811yyda}, We have
\begin{equation}\label{e-215}
\begin{split}
L_0 u&=u(0)=a(0,0,m)\ol{\chi_k(0)}V(0)\tau(0)\\
&=d_kV(0)\big(a_0(0)m^n+a_1(0)m^{n-1} +a_2(0)m^{n-2} +\cdots \big).
\end{split}
\end{equation}
Then 
\begin{equation}\label{e-216}
b_{0,k}(p)=\pi^{\frac{d}{2}}d_k^2\frac{a_0(p)}{V_{{\rm eff\,}}(p)}2^{\frac{d}{2}}
=2^{\frac{d}{2}-1}\pi^{\frac{d}{2}-n-1}\frac{d_k^2}{V_{{\rm eff\,}}(p)}.
\end{equation}

Let $j=1$, then $\nu=1,\mu=0$ or $\nu=2,\mu=1$. Set 
\begin{equation}\label{e-gue200812ycda}
\begin{split}
&m^{-1}L_1 u:=A_1+A_2,\\
&A_1=i^{-1}2^{-1}(\langle\,F'(0)D\,,\,D\,\rangle u)(0),\\
&A_2=i^{-1}2^{-2}(\langle\,F'(0)D\,,\,D\,\rangle)^2(\frac{hu}{2})(0).
\end{split}
\end{equation}
Recall that $(\frac{\pr}{\pr y_j}V)(0)=0$, for $j=1,...,d$. When $\nu=1, \mu=0$, 
\begin{equation}\label{e-217}
\begin{split}
A_1=&\frac{1}{4m}\Delta u(0)=\frac{1}{4m}\sum_{j=1}^d\frac{\pr^2}{\pr y_j^2}
\big(a(y',0,m) V(y')\ol{\chi_k(y')}\tau(y') \big)(0)\\
=&\frac{1}{4m}\sum_{j=1}^d\big(
a_0(0) \frac{\pr^2 V}{\pr y_j^2}(0)d_km^n+a_0(0)V(0)\frac{\pr^2 \ol{\chi_k}}{\pr y_j^2}(0)m^n\big)+O(m^{n-2}).
\end{split}
\end{equation}
When $\nu=2$, by using Proposition \ref{p-1910}, we deduce that
\begin{equation}\label{e-218}
\begin{split}
A_2=&\frac{i}{32m}\Delta^2(hu)(0)=\frac{i}{32m}\sum_{j,\ell=1}^d\frac{\pr^4(hu)}{\pr y_j^2\pr y_\ell^2}(0)
=\frac{i}{32m}\sum_{j,\ell=1}^d\frac{\pr^4h}{\pr y_j^2\pr y_\ell^2}(0) u(0)\\
=&\frac{i}{32m}\sum_{j,\ell=1}^d\frac{\pr^4h}{\pr y_j^2\pr y_\ell^2}(0)a_0(0) V(0)d_km^n+O(m^{n-2}).
\end{split}
\end{equation}
Combining \eqref{e-215}, \eqref{e-217} and \eqref{e-218}, we get
\begin{equation}\label{e-gue200816ycdh}
\begin{split}
b_{1,k}(p)=&\pi^{\frac{d}{2}}d_k\big(a_1(0) d_kV(0)+\frac{1}{4}a_0(0)V(0)\Delta\ol{\chi_k}(0)+\\
&\frac{1}{4}a_0(0)d_k\Delta V(0)+
\frac{i}{32}a_0(0)V(0)d_k\Delta^2h(0) \big).
\end{split}
\end{equation}

We need 

\begin{proposition}\label{p-gue200824yyd}
\begin{equation}\label{e-gue200824yyd}
\Delta V(0) = 2^{\frac{d}{2}-2}\frac{1}{V_{{\rm eff\,}}(p)}\sum_{s, j=1}^d \frac{\pr^2}{\pr y_s^2} \langle\, \frac{\pr}{\pr y_j} \,|\, \frac{\pr}{\pr y_j} \,\rangle(p).
\end{equation}
\end{proposition}

\begin{proof}
Recall that, in $y'$-coordinates, we write $d\mu(g)=V(y')dy'$. 
From \eqref{e-gue200811yyd} and Taylor formula, we have 
\[
\begin{split}
\Delta V(y') = & \Delta \left( \frac{1}{V_{{\rm eff\,}}(p)}\left( \det \langle\, \frac{\pr}{\pr y_j} \,|\, \frac{\pr}{\pr y_\ell} \,\rangle^d_{j, \ell=1}  \right)^{\frac{1}{2}} \right)        \\ 
= & \frac{1}{V_{{\rm eff\,}}(p)} \sum_{s=1}^d \frac{\pr^2}{\pr y_s^2} \left(\exp \frac{1}{2} \log \det \langle\, \frac{\pr}{\pr y_j} \,|\, \frac{\pr}{\pr y_\ell} \,\rangle^d_{j, \ell=1}   \right)  \\
= & \frac{1}{V_{{\rm eff\,}}(p)} \sum_{s=1}^d \frac{\pr^2}{\pr y_s^2} \left(\exp \frac{1}{2} \sum_{j=1}^d \left( \log \langle\, \frac{\pr}{\pr y_j} \,|\, \frac{\pr}{\pr y_j} \,\rangle + O(|y'|^4) \right) \right)  \\
= &\frac{1}{V_{{\rm eff\,}}(p)} \exp \frac{1}{2} \sum_{j=1}^d \left( \log \langle\, \frac{\pr}{\pr y_j} \,|\, \frac{\pr}{\pr y_j} \,\rangle + O(|y'|^4) \right) \\
& \cdot  \frac{1}{2} \left(  \sum_{s, j=1}^d  \frac{\pr^2}{\pr y_s^2} \langle\, \frac{\pr}{\pr y_j} \,|\, \frac{\pr}{\pr y_j} \,\rangle  \langle\, \frac{\pr}{\pr y_j} \,|\, \frac{\pr}{\pr y_j} \,\rangle^{-1} + O(|y'|^2) \right) \\
= &\frac{1}{2} \frac{1}{V_{{\rm eff\,}}(p)} \left( \det \langle\, \frac{\pr}{\pr y_j} \,|\, \frac{\pr}{\pr y_\ell} \,\rangle^d_{j, \ell=1} \right)^{\frac{1}{2}} \\
&  \cdot   \left(  \sum_{s, j=1}^d  \frac{\pr^2}{\pr y_s^2}  \langle\, \frac{\pr}{\pr y_j} \,|\, \frac{\pr}{\pr y_j} \,\rangle  \langle\, \frac{\pr}{\pr y_j} \,|\, \frac{\pr}{\pr y_j} \,\rangle^{-1} + O(|y'|^2) \right). 
\end{split}
\]
Note that $\langle\, \frac{\pr}{\pr y_j} \, | \, \frac{\pr}{\pr y_\ell} \rangle = 2\delta_{j, \ell} + O(|y'|^2)$, $j, \ell=1, \cdots, d$.
Therefore, we derive
\[
\Delta V(0) = 2^{\frac{d}{2}-2}\frac{1}{V_{{\rm eff\,}}(p)}\sum_{s, j=1}^d \frac{\pr^2}{\pr y_s^2} \langle\, \frac{\pr}{\pr y_j} \,|\, \frac{\pr}{\pr y_j} \,\rangle(p).
\]
\end{proof}

\section{Global expression for coefficients of lower order terms}
We now express $b_{1,k}$ as geometric invariants of $X$ and $G$. 
From now on, we fix $p\in\mu^{-1}(0)$ and we will use the same notations as before. From \eqref{e-gue200811yyd}, we see that the orthonormal basis for $T_{e_0}G$ is 
\begin{equation}\label{e-gue200816ycda}
\{\frac{1}{\sqrt{2}}V_{{\rm eff\,}}(p)^{\frac{1}{d}}\frac{\pr}{\pr y_1},...,\frac{1}{\sqrt{2}}V_{{\rm eff\,}}(p)^{\frac{1}{d}}\frac{\pr}{\pr y_d}\}.
\end{equation} 
Let $\Delta_{d\mu}$ be the de Rham Laplacian for functions with respect to $\langle\,\cdot\,|\,\cdot\,\rangle_{d\mu}$. From \eqref{e-gue200816ycda} and \eqref{e-gue200811yyd},
we see that 
\begin{equation}\label{e-224}
\Delta=2V_{{\rm eff\,}}(p)^{-\frac{2}{d}}\Delta_{d\mu}+O(|y'|).
\end{equation}
Recall that $\Delta$ is given by \eqref{e-gue200811yydh}. From \eqref{e-1912}, \eqref{e-gue200811yyda} and \eqref{e-224}, we get 
\begin{equation}\label{e-gue200816ycdb}
\pi^{\frac{d}{2}}d_k\frac{1}{4}a_0(0)V(0)(\Delta\ol\chi_k)(0)= \frac{1}{4} \pi^{\frac{d}{2}-n-1}2^{\frac{d}{2}} \frac{d_k}{ V_{{\rm eff\,}}(p)^{1+\frac{2}{d}}} (\Delta_{d\mu}\ol\chi_k)(e_0).
\end{equation}
Now we compute $\Delta^2h(0)$ in $b_{1,k}$.
Recall that $\Phi(y',0)=\Phi(x,0)=i\varphi(x)=i\varphi(y')+O(\abs{y'}^N)$, for every $N\in\mathbb N$. 
Then
\begin{equation}\label{e-05151}
\Delta^2h(p)=i\Delta^2\varphi.
\end{equation}
It suffices to compute 
\begin{equation*}
\Delta^2\varphi(0)=\sum_{j,\ell=1}^d\frac{\pr^4\varphi}{\pr y_j^2\pr y_\ell^2}(0).
\end{equation*} 
For $a_{j,\ell}$, $b_j$, $j=1,\ldots,d$, $\ell=1,\ldots,n$, in \eqref{e-gue200812yydI}, it is easy to see that at $(y_1,\ldots,y_d,0,\ldots,0)$,  $a_{j,\ell}$, $b_j$, $j,1,\ldots,d$, $\ell=1,\ldots,n$, are independent of $x_{2n+1}$.
From now on, we will use Einstein summation convention. We have 
\begin{equation*}\label{e-710}
\begin{split}
\frac{\pr^2}{\pr y_j^2}=&(\frac{\pr}{\pr z_j}+\frac{\pr}{\pr\bar z_j}+ a_{j,\alpha}\frac{\pr}{\pr z_\alpha}+\bar a_{j,\alpha}\frac{\pr}{\pr \bar z_\alpha}+b_j \frac{\pr}{\pr x_{2n+1}}   )\\
&(\frac{\pr}{\pr z_j}+\frac{\pr}{\pr\bar z_j}+ a_{j,\beta}\frac{\pr}{\pr z_\beta}+\bar a_{j,\beta}\frac{\pr}{\pr \bar z_\beta}+b_j \frac{\pr}{\pr x_{2n+1}}   )\\
=&\frac{\pr^2}{\pr z_j^2}+\frac{\pr^2}{\pr z_j\pr \bar z_j}+\frac{\pr a_{j,\beta}}{\pr z_j}\frac{\pr}{\pr z_\beta}+\frac{\pr \bar a_{j,\beta}}{\pr z_j}\frac{\pr}{\pr\bar z_\beta}+
a_{j,\beta}\frac{\pr^2}{\pr z_j\pr z_\beta}+\bar a_{j,\beta}\frac{\pr^2}{\pr z_j\pr\bar z_\beta}\\
&+\frac{\pr^2}{\pr z_j\pr \bar z_j}+\frac{\pr^2}{\pr \bar z_j^2}+\frac{\pr a_{j,\beta}}{\pr \bar z_j}\frac{\pr}{\pr z_\beta}+\frac{\pr \bar a_{j,\beta}}{\pr\bar z_j}\frac{\pr}{\pr\bar z_\beta}+a_{j,\beta}\frac{\pr^2}{\pr\bar z_j\pr z_\beta}+\bar a_{j,\beta}\frac{\pr^2}{\pr \bar z_j\pr\bar z_\beta}\\
&+a_{j,\alpha}\bigl(\frac{\pr^2}{\pr z_j\pr z_\alpha}+\frac{\pr^2}{\pr z_\alpha\pr \bar z_j}+\frac{\pr a_{j,\beta}}{\pr z_\alpha}\frac{\pr}{\pr z_\beta}+\frac{\pr \bar a_{j,\beta}}{\pr z_\alpha}\frac{\pr}{\pr\bar z_\beta}+
a_{j,\beta}\frac{\pr^2}{\pr z_\alpha\pr z_\beta}+\bar a_{j,\beta}\frac{\pr^2}{\pr z_\alpha\pr\bar z_\beta}\bigr)\\
&+\bar a_{j,\alpha}\bigl(\frac{\pr^2}{\pr z_j\pr \bar z_\alpha}+\frac{\pr^2}{\pr \bar z_j\pr \bar z_\alpha}+\frac{\pr a_{j,\beta}}{\pr \bar z_\alpha}\frac{\pr}{\pr z_\beta}+\frac{\pr \bar a_{j,\beta}}{\pr\bar z_\alpha}\frac{\pr}{\pr\bar z_\beta}+a_{j,\beta}\frac{\pr^2}{\pr\bar z_\alpha\pr z_\beta}+\bar a_{j,\beta}\frac{\pr^2}{\pr \bar z_\alpha\pr\bar z_\beta}   \bigr)\\
&+Q,
\end{split}
\end{equation*}
where $Q$ denotes the sum of all the terms involving $b_j$ and $\frac{\pr}{\pr x_{2n+1}}$.
From now on, we always take values at the point $p$ during computations, where we ignore $p$ with no confusion.
By direct computation and Proposition \ref{p-gue200817yyd}, we have
\begin{equation}\label{e-712}
\begin{split}
\frac{\pr^4\varphi}{\pr y_j^2\pr y_\ell^2}=&\bigl(\frac{\pr^4\varphi}{\pr z_j^2\pr\bar z_\ell^2}+4\frac{\pr^4\varphi}{\pr z_j\pr \bar z_j\pr z_\ell\pr\bar z_\ell}+\frac{\pr^4\varphi}{\pr \bar z_j^2\pr z_\ell^2}    \bigr)\\
&+2\bigl(\frac{\pr^2 a_{j,j}}{\pr z_\ell^2}+\frac{\pr^2 \bar a_{j,j}}{\pr\bar  z_\ell^2}+\frac{\pr^2 a_{j,j}}{\pr\bar z_\ell^2}+\frac{\pr^2\bar a_{j,j}}{\pr z_\ell^2}   \bigr)+4\bigl(\frac{\pr^2 a_{j,j}}{\pr z_\ell\pr\bar z_\ell}+ \frac{\pr^2 \bar a_{j,j}}{\pr z_\ell\pr\bar z_\ell}   \bigr)\\
&+2\bigl(\frac{\pr^2 a_{j,\ell}}{\pr z_j\pr  z_\ell}+\frac{\pr^2 a_{j,\ell}}{\pr\bar  z_j\pr \bar  z_\ell}+\frac{\pr^2\bar  a_{j,\ell}}{\pr\bar  z_j\pr  z_\ell}+\frac{\pr^2 a_{j,\ell}}{\pr z_j\pr\bar  z_\ell}\\
&+\frac{\pr^2\bar  a_{j,\ell}}{\pr z_j\pr  z_\ell}+\frac{\pr^2\bar  a_{j,\ell}}{\pr\bar  z_j\pr \bar  z_\ell}
+\frac{\pr^2\bar a_{j,\ell}}{\pr z_j\pr\bar  z_\ell}+\frac{\pr^2 a_{j,\ell}}{\pr\bar  z_j\pr  z_\ell}   \bigr).
\end{split}
\end{equation}
Note that
\begin{equation}\label{e-05152}
\frac{\pr^2 a_{j,j}}{\pr y_\ell^2}= \frac{\pr^2 a_{j,j}}{\pr z_\ell^2}+2\frac{\pr^2 a_{j,j}}{\pr z_\ell\pr\bar z_\ell}+\frac{\pr^2 a_{j,j}}{\pr\bar  z_\ell^2}
\end{equation}
and
\begin{equation}\label{e-05153}
\frac{\pr^2 a_{j,\ell}}{\pr y_j\pr y_\ell}=\frac{\pr^2 a_{j,\ell}}{\pr z_j\pr z_\ell}+\frac{\pr^2 a_{j, \ell}}{\pr z_j\pr \bar z_\ell}+\frac{\pr^2 a_{j, \ell}}{\pr\bar  z_j\pr z_\ell}+\frac{\pr^2 a_{j,\ell}}{\pr \bar z_j\pr\bar  z_\ell}.
\end{equation}
Then we deduce from \eqref{e-712}, \eqref{e-05152} and \eqref{e-05153} that
\begin{equation}\label{e-713}
\begin{split}
\frac{\pr^4\varphi}{\pr y_j^2\pr y_\ell^2}=&\bigl(\frac{\pr^4\varphi}{\pr z_j^2\pr\bar z_\ell^2}+4\frac{\pr^4\varphi}{\pr z_j\pr \bar z_j\pr z_\ell\pr\bar z_\ell}+\frac{\pr^4\varphi}{\pr \bar z_j^2\pr z_\ell^2}    \bigr)\\
&+2\frac{\pr^2}{\pr y_\ell^2}(a_{j,j}+\bar a_{j,j})+2\frac{\pr^2}{\pr y_j\pr y_\ell}( a_{j,\ell}+\bar  a_{j,\ell}).  
\end{split}
\end{equation}

By the definition of the Levi metric, 
\begin{equation*}\label{e-714}
\begin{split}
\langle \frac{\pr}{\pr y_j}| \frac{\pr}{\pr y_\ell}\rangle=&\langle \hat a_{j,\alpha}\bigl( \frac{\pr}{\pr z_\alpha} +i\frac{\pr \varphi}{\pr z_\alpha}\frac{\pr}{\pr x_{2n+1}} \bigr)+
\ol{\hat a_{j,\alpha}}\bigl( \frac{\pr}{\pr\bar z_\alpha} -i\frac{\pr \varphi}{\pr\bar z_\alpha}\frac{\pr}{\pr x_{2n+1}} \bigr)\big| \\
& \hat a_{\ell,\beta}\bigl( \frac{\pr}{\pr z_\beta} +i\frac{\pr \varphi}{\pr z_\beta}\frac{\pr}{\pr x_{2n+1}} \bigr)+
\ol{\hat a_{\ell, \beta}}\bigl( \frac{\pr}{\pr\bar z_\beta} -i\frac{\pr \varphi}{\pr\bar z_\beta}\frac{\pr}{\pr x_{2n+1}} \bigr)\rangle\\
=&\frac{\pr^2\varphi}{\pr z_j\pr\bar z_\ell}+\frac{\pr^2\varphi}{\pr\bar z_j\pr z_\ell}+\bar a_{\ell, \beta}\frac{\pr^2\varphi}{\pr z_j\pr\bar z_\beta}+ a_{\ell,\beta}\frac{\pr^2\varphi}{\pr\bar z_j\pr z_\beta}\\
&+a_{j,\alpha}\frac{\pr^2\varphi}{\pr\bar z_\ell\pr z_\alpha}+\bar a_{j,\alpha}\frac{\pr^2\varphi}{\pr z_\ell\pr\bar  z_\alpha}+
a_{j,\alpha}\bar a_{\ell,\beta}\frac{\pr^2\varphi}{\pr z_\alpha\pr\bar  z_\beta}+\bar a_{j,\alpha}a_{\ell,\beta}\frac{\pr^2\varphi}{\pr\bar z_\alpha\pr  z_\beta}.
\end{split}
\end{equation*}
Then
\begin{equation}\label{e-715}
\frac{\pr^2}{\pr y_\ell^2}\langle \frac{\pr}{\pr y_j}| \frac{\pr}{\pr y_j}\rangle=4\frac{\pr^4\varphi}{\pr z_j\pr\bar z_j\pr z_\ell\pr\bar z_\ell}+2\frac{\pr^2}{\pr y_\ell^2}(a_{j,j}+\bar a_{j,j}), 
\end{equation}
and
\begin{equation}\label{e-716}
\begin{split}
\frac{\pr^2}{\pr y_j\pr y_\ell} \langle \frac{\pr}{\pr y_j}| \frac{\pr}{\pr y_\ell}\rangle=&\frac{\pr^4\varphi}{\pr z_j^2\pr\bar z_\ell^2}+2\frac{\pr^4\varphi}{\pr z_j\pr \bar z_j\pr z_\ell \pr\bar z_\ell}+\frac{\pr^4\varphi}{\pr \bar z_j^2\pr z_\ell^2} \\
&+\frac{\pr^2}{\pr y_j\pr y_{\ell}} (a_{j,\ell}+\bar a_{j,\ell}+a_{\ell,j}+\bar a_{\ell,j}).   
\end{split}
\end{equation}
By \eqref{e-713}, \eqref{e-715} and \eqref{e-716}, it turns out that
\begin{equation}\label{e-717}
\begin{split}
\frac{\pr^4\varphi}{\pr y_j^2\pr y_\ell^2}=&\frac{1}{2}\frac{\pr^2}{\pr y_j^2}\langle \frac{\pr}{\pr y_\ell}| \frac{\pr}{\pr y_\ell}\rangle+\frac{1}{2}\frac{\pr^2}{\pr y_\ell^2}\langle \frac{\pr}{\pr y_j}| \frac{\pr}{\pr y_j}\rangle\\
&+\frac{\pr^2}{\pr y_j\pr y_\ell}\langle \frac{\pr}{\pr y_j}| \frac{\pr}{\pr y_\ell}\rangle-2\frac{\pr^4\varphi}{\pr z_j\pr \bar z_j\pr z_\ell\pr\bar z_\ell}. 
\end{split}
\end{equation}

Let $h=(h_{j,\ell})^n_{j,\ell=1}$, where $h_{j,\ell}=\langle Z_j| Z_\ell \rangle$, $h^{-1}=(h^{j,\ell})^n_{j,\ell=1}$ and $\theta=(\theta_{j,\ell})=(\pr_b h)h^{-1}$. The Levi metric induces the rigid Chern connection $\nabla$ as follows
\begin{equation*}
	\begin{split}
	\nabla: T^{1,0}X&\rightarrow CT^{\ast}X\otimes T^{1,0}X, \\
	\sum c_j Z_j&\mapsto\sum(dc_j)Z_j+\sum a_j\theta_{j,\ell}Z_l.
	\end{split}
\end{equation*}

\begin{definition}\label{d-416}
The rigid Chern curvature of the connection $\bigtriangledown$ is
\begin{equation*}
	R^{HX}:=\bar{\pr_b}((\pr_b h)h^{-1})=(R_{j,\ell})\in C^\infty (X, T^{\ast 1,1}X\otimes {\rm End}(T^{1,0}X)).
\end{equation*}
\end{definition}
Given $U, V\in T^{1,0}X$, $\eta=\sum\eta_j Z_j\in T^{1,0}X$, we have
\begin{equation*}
	R^{HX}(\bar U,V)\eta=\sum_{j,\ell=1}^n\langle R_{j,\ell}, \bar U\wedge V\rangle \eta_\ell Z_j.
\end{equation*} 

\begin{definition}\label{d-gue200914yyd}
For $x\in\mu^{-1}(0)$, let $\{u_1,\ldots,u_d\}$ be a basis for $\underline{\mathfrak{g}}_x$ such that $\{u_1-iJu_1,\ldots,u_d-iJu_d\}$ is an orthonormal
basis for $\underline{\mathfrak{g}}_x-iJ\underline{\mathfrak{g}}_x$ with respect to $\langle\,\cdot\,|\,\cdot\,\rangle$. The scalar curvature of $X$ in the direction of $G$ is defined by $R_e(x):=\sum^d_{j,\ell=1}\langle\,R^{HX}_x(\bar e_j,e_\ell)e_j\,|\,e_\ell\,\rangle$, where $e_j=u_j-iJu_j$, $j=1,\ldots,d$. 
\end{definition}

It is straightforward to see that Definition~\ref{d-gue200914yyd} is well-defined. 
Let $u_j=\frac{\pr}{\pr y_j}(p)$, $j=1,\ldots,d$. Then, $e_j:=\frac{1}{2}(\frac{\pr}{\pr y_j}(p)-i(J\frac{\pr}{\pr y_j})(p))$, $j=1,\ldots,d$, is an orthonormal basis for 
$\underline{\mathfrak{g}}_p-iJ\underline{\mathfrak{g}}_p$.

\begin{lemma}\label{l-416}
With the notations used above, we have 
\begin{equation*}
	\langle R^{HX}_p(\bar e_s, e_t)e_j|e_\ell\rangle=\frac{\pr^4 \varphi}{\pr\bar z_s\pr z_t\pr\bar z_\ell\pr z_j}(p).
\end{equation*}
Thus, $R_e(p)=\sum_{j,\ell=1}^d \langle R^{HX}_p(\bar e_j, e_\ell)e_j|e_\ell\rangle$.
\end{lemma}
\begin{proof}
By Proposition \ref{p-1910}, we have $\varphi=o(|y'|^4)$.
Then
\begin{equation*}
	\begin{split}
		\bar{\pr_b}((\pr_b h)h^{-1})(p)&=\bar{\pr_b}\bigl((\frac{\pr^3 \varphi}{\pr z_t\pr\bar z_\ell\pr z_j}dz_t)\cdot (\frac{\pr^2 \varphi}{\pr\bar z_\ell\pr z_j})^{-1}   \bigr)(p)\\
		&=\bigl( \frac{\pr^4 \varphi}{\pr\bar z_s\pr z_t\pr\bar z_\ell\pr z_j}d\bar z_s\wedge dz_t   \bigr)(p)\cdot (\frac{\pr^2 \varphi}{\pr\bar z_\ell\pr z_j})^{-1}(p)\\
		&=\bigl( \frac{\pr^4 \varphi}{\pr\bar z_s\pr z_t\pr\bar z_\ell\pr z_j}d\bar z_s\wedge dz_t   \bigr)(p).
	\end{split}
\end{equation*}
Hence
\begin{equation*}
	\langle R^{HX}_p(\bar Z_s, Z_t)Z_j|Z_\ell\rangle(p)=\frac{\pr^4 \varphi}{\pr\bar z_s\pr z_t\pr\bar z_\ell\pr z_j}(p).
\end{equation*}
The lemma is proved following the fact that $e_j(p)=Z_j(p)$, $j=1,\ldots,d$. 
\end{proof}

Recall that the Christoffel symbols of $G$ with respect to $\langle\,\cdot\,|\,\cdot\,\rangle_{d\mu}$ in local coordinates $y'=(y_1,\ldots,y_d)$ are given by 
\[\Gamma_{s,t}^\ell=\frac{1}{2}g^{\ell ,j}\bigl( \frac{\pr g_{j,s}}{\pr y_t}+\frac{\pr g_{j,t}}{\pr y_s}-  \frac{\pr g_{s,t}}{\pr y_j}\bigr),\ \ \ell, s, t=1,\ldots,d,\]
where $g_{j,\ell}=\langle\,\frac{\pr}{\pr y_j}\,|\,\frac{\pr}{\pr y_\ell}\,\rangle_{d\mu}$, $j, \ell=1,\ldots,d$, and $\left(g^{j,\ell}\right)^d_{j,\ell=1}$ is the inverse matrix 
of the matrix $\left(g_{j,\ell}\right)^d_{j,\ell=1}$. The Ricci curvature tensor of $G$ with respect to $\langle\,\cdot\,|\,\cdot\,\rangle_{d\mu}$ in local coordinates $y'=(y_1,\ldots,y_d)$ is given by 
\[R_{j,\ell}:=\sum^d_{a=1}\frac{\pr\Gamma^a_{j,\ell}}{\pr y_a}-\sum^d_{a=1}\frac{\pr\Gamma^a_{a,\ell}}{\pr y_j}+\sum^d_{a,b=1}\Bigr(\Gamma^a_{a,b}\Gamma^b_{j,\ell}-\Gamma^a_{j,b}\Gamma^b_{a,\ell}\Bigr),\ \ j, \ell=1,\ldots,d.\]
The scalar curvature $S_G$ on $G$ with respect to $\langle\,\cdot\,|\,\cdot\,\rangle_{d\mu}$ is defined by 
\begin{equation}\label{e-gue200913yyd}
S_G=g^{j,k}R_{j,k}. 
\end{equation}
We can check that 
\begin{equation*}
	\Gamma_{s,t}^\ell(p)=\frac{1}{2}h^{\ell ,j}\bigl( \frac{\pr h_{j,s}}{\pr y_t}+\frac{\pr h_{j,t}}{\pr y_s}-  \frac{\pr h_{s,t}}{\pr y_j}\bigr)(p)=0,
\end{equation*}
where $h_{j,\ell}=\langle \frac{\pr}{\pr y_j}| \frac{\pr}{\pr y_\ell}\rangle$. Moreover, it is straightforward to see that 
\begin{equation}\label{e-417}
S_G(p)=(V_{{\rm eff\,}}(p))^{\frac{2}{d}}\frac{1}{4}\sum_{j,\ell=1}^d\bigl( \frac{\pr^2}{\pr y_j\pr y_\ell}\langle \frac{\pr}{\pr y_j}| \frac{\pr}{\pr y_\ell}\rangle-\frac{\pr^2}{\pr y_\ell^2}\langle \frac{\pr}{\pr y_j}| \frac{\pr}{\pr y_j}\rangle \bigr)(p).
\end{equation}

It follows from \eqref{e-717}, \eqref{e-417} and Lemma \ref{l-416} that

\begin{equation}\label{e-418}
\Delta^2\varphi(p)=2\sum_{j,\ell=1}^d \frac{\pr^2}{\pr y_j^2} \langle\, \frac{\pr}{\pr y_\ell}  \,|\, \frac{\pr}{\pr y_\ell} \, \rangle+4(V_{{\rm eff\,}}(p))^{-\frac{2}{d}}S_G(p)-2R_e(p).
\end{equation}
By 
\eqref{e-1912}, \eqref{e-gue200816ycdh}, \eqref{e-gue200824yyd}, \eqref{e-gue200816ycdb}, \eqref{e-05151} and \eqref{e-418}, we finally obtain
\begin{equation}\label{e-0519}
\begin{split}
b_{1,k}(p)=&
\frac{1}{4}\pi^{\frac{d}{2}-n-1}2^{\frac{d}{2}}   \frac{d_k^2}{V_{{\rm eff\,}}(p)}R(p)\\
&+ \frac{1}{4}\pi^{\frac{d}{2}-n-1}2^{\frac{d}{2}}  \frac{d_k}{V_{{\rm eff\,}}(p)^{1+\frac{2}{d}}} \Delta_{d\mu}\bar\chi_k(p)\\
&-\frac{1}{64}\pi^{\frac{d}{2}-n-1} 2^{\frac{d}{2}} \frac{d_k^2}{V_{{\rm eff\,}}(p)}(4(V_{{\rm eff\,}}(p))^{-\frac{2}{d}}S_G(p)-2R_e(p)).
\end{split}
\end{equation}
From \eqref{e-0519}, Theorem~\ref{t-gue170128I} follows then.



\bibliographystyle{plain}

\end{document}